\documentclass[a4 paper, 11pt]{article}

\usepackage{amssymb, amsmath, indentfirst, amsfonts, latexsym, amsthm, fancyhdr, graphicx, rotating, array, breqn, accents, bm, booktabs}
\usepackage{arydshln}
\usepackage{pdflscape}

\usepackage[english]{babel}
\usepackage[hidelinks]{hyperref}

\usepackage{pgf,tikz}
\usetikzlibrary{arrows}

\usepackage[utf8]{inputenc}
\usepackage[T1]{fontenc}

\newcommand{\R}{\mbox{$\mathbb{R}$}}

\newcommand{\arr}{\mbox{$\Longrightarrow$}}
\newcommand{\rar}{\mbox{$\longrightarrow$}}

\newcommand{\vspa}{\vspace{.5cm}} 
\newcommand{\spa}{\mbox{$\hspace{1em}$}} 
\newcommand{\spaa}{\mbox{$\hspace{2em}$}} 
\newcommand{\spaaa}{\mbox{$\hspace{3em}$}}

\newcommand{\al}{\alpha}

\newtheorem{theo}{Theorem}[section]

\newtheorem{defi}[theo]{Definition}

\newtheorem{rem}[theo]{Remark}

\newtheorem{ex}[theo]{Example}



 \newenvironment{MYitemize}
{
\begin{itemize}}
{\end{itemize}}

\newcommand{\tabfrac}[2]{%
        \setlength{\fboxrule}{0pt}%
        \fbox{$\dfrac{#1}{#2}$}%
}

\title{A Characterization  of a Subclass of Separate Ratio-Type Copulas}

\date{\today}

\author{ 
\thanks{Liwa College, Abu Dhabi, UAE. Email: ziad.adwan@lc.ac.ae} Ziad Adwan
\thanks{New York University, Abu Dhabi, UAE.
Email: ns6159@nyu.edu} Nicola Sottocornola
}

\begin{document}

\selectlanguage{english}

\maketitle

\begin{abstract}
Copulas are essential tools in statistics and probability theory, enabling the study of the dependence structure between random variables independently of their marginal distributions. Among the various types of copulas, Ratio-Type Copulas have gained significant attention due to their flexibility in modeling joint distributions. This paper focuses on Separate Ratio-Type Copulas, where the dependence function is a separate product of univariate functions.

We revisit a theorem characterizing the validity of these copulas under certain assumptions, generalize it to broader settings, and examine the conditions for reversing the theorem in the case of concave generating functions. To address its limitations, we propose new assumptions that ensure the validity of separate copulas under specific conditions. These results refine the theoretical framework for separate copulas, extending their applicability to pure mathematics and applied fields such as finance, risk management, and machine learning.
\end{abstract}




\section{Introduction}	                                    


Copulas are powerful tools in statistics and science for modeling and analyzing complex relationships between random variables. Unlike traditional methods that focus on linear correlations, copulas comprehensively capture dependencies, including tail dependence and asymmetries. This makes them particularly valuable in fields where understanding intricate interdependencies is critical \cite{N}.

One of the primary advantages of copulas is their ability to separate the marginal distributions of random variables from their dependence structure. This flexibility enables researchers to model the behavior of individual variables using appropriate distributions while independently specifying their dependence \cite{S}. For instance, in finance, copulas are used to analyze the joint risk of assets, allowing for better portfolio optimization and risk management \cite{EMS}.

In environmental science, copulas help model the interplay between variables like rainfall and temperature, enabling more accurate predictions of extreme weather events \cite{GF}. Similarly, in medicine, they are used to study the correlation between biomarkers and disease outcomes, advancing personalized treatment plans \cite{J}.

Beyond applied sciences, copulas play a key role in machine learning, reliability analysis, and econometrics. By capturing the essence of dependence structures, copulas provide a versatile framework for addressing real-world problems characterized by uncertainty and interdependence, making them indispensable in modern statistical analysis.
\vspa

Formally, a bivariate copula \(C(u, v)\) is a function that maps the unit square \(S = [0, 1]^2\) to \(I = [0, 1]\), satisfying the following conditions:
\begin{enumerate}
    \item \(C(u, 0) = C(0, v) = 0\) for all \((u, v) \in S\),
    \item \(C(u, 1) = u\) and \(C(1, v) = v\) for all \(u, v \in I\),
    \item For any \((u_1, v_1), (u_2, v_2) \in S\) such that \(u_1 \leq u_2\) and \(v_1 \leq v_2\),
    \begin{equation}\label{eq_C}
    C(u_2, v_2) - C(u_2, v_1) - C(u_1, v_2) + C(u_1, v_1) \geq 0.
    \end{equation}
\end{enumerate}

These properties ensure that a copula captures the dependency structure of a joint distribution while being independent of the marginal distributions of the random variables.

One particularly flexible family of copulas is the \textit{Ratio-Type Copulas}, which are defined as:
$$
B_\theta(u,v) = \frac{uv}{1 - \theta \phi(u,v)}\,, \spaa 0 \leq u,\, v \leq 1 \, ,\,  \theta \in \R
$$
where $\phi$ is a real function defined on $S$. In recent years they attracted considerable attention (\cite{C1}, \cite{C2}, \cite{C3}, \cite{KBM}, \cite{MS}). To simplify the analysis, researchers have focused on \textit{Separate Ratio-Type Copulas}, a subclass where the dependence function \(\phi(u, v)\) is expressed as a separable product of two univariate functions  that we assume differentiable a.e.:
\begin{equation} \label{eq_D}
D_\theta(u,v) =  \frac{uv}{1 - \theta f(u) g(v)}\,, \spaa 0 \leq u,\, v \leq 1 \, ,\,  \theta \in \R.
\end{equation}

Now let's consider the function $G$ defined on $S$ in this way:
\begin{equation}\label{eq_G}
G= (f-uf')(g-vg') - 2 uv f' g'
\end{equation}
and define 
\begin{equation}\label{eq_alpha}
\alpha_1 = \underset{S}{\min} (G) \spaaa
\alpha_2 = \underset{S}{\max} (G)\,.
\end{equation}

In a nice paper published in 2024 \cite{KBM}, El Ktaibi, Bentoumi and Mesfioui examine the conditions under which  \eqref{eq_D} is a valid copula. Given the assumptions:

\begin{itemize}

\item[A1.]  $f(1) = g(1) = 0$.

\item[A2.] $f$ and $g$ are strictly monotonic functions.

\item[A3.] $\dfrac{f(u) g(v)}{f(0) g(0)} \leq 1 - uv$.

\end{itemize}
they proved that $D_\theta$ is a valid copula provided that $1/\alpha_1 \leq \theta \leq 1/\alpha_2$.

The aim of this paper is:

\begin{MYitemize}

\item To prove that under the Assumptions A1, A2, A3  the theorem cannot be reversed.

\item To provide additional Assumptions to make it possible.

\end{MYitemize}


\section{Bivariate copulas}	                                    


Before moving forward, we recall the general definition of a bivariate copula $C$, assuming that all the derivatives involved exist a.e. (\cite{N}, \cite{KBMS}):

\begin{defi}
The function $C\ :\ S\ \rar\ I$ is a bivariate copula if:
\begin{enumerate}

\item $C(u,0) = C(0,v) = 0,\  C(u,1)=u,\  C(1,v)=v,\ \forall \, u, v \in I$

\item $\dfrac{\partial^2 C}{\partial u \partial v} \geq 0, \ \forall \,  (u,v) \in S$

\end{enumerate}
\end{defi}

Here inequality \eqref{eq_C} has been replaced with the more comfortable condition 2. on the second derivative. In the case of copulas of the form \eqref{eq_D} the first condition is trivially verified. The second one reduces to (see \cite{KBM}):
$$
\frac{1 - \theta \left[ (f-uf')(g-vg') -2D_\theta f' g' \right] }{(1- \theta fg)^2  } \geq 0
$$
or, which is the same, to
\begin{equation}\label{ineq}
1 - \theta \left[ (f-uf')(g-vg') -2D_\theta f' g' \right]  \geq 0
\end{equation}


\section{The theorem}	                                    


We start this section with a couple of simple observations:

\begin{rem}\label{rem2}
We can assume, without loss of generality, that $f(0) = g(0) =1$.
\end{rem}
\begin{proof}
It's enough to remark that, if  \eqref{eq_D} is a valid copula, so is  $\tilde{D} = uv / (1 - \tilde{\theta} \tilde{f} \tilde{g})$ with $\tilde{f} = f/f(0),\ \tilde{g} = g/g(0)$ and $\tilde{\theta} = f(0) g(0) \theta$. 
\end{proof}

\begin{rem}
If $M = \max(f-uf')$ and $N = \max(g-vg')$ then $\underset{\partial S}{\max} (G) = \max \{M,N \}$.
\end{rem}
\begin{proof}
Let's find the maximum of $G$ on $\partial S$:
\begin{eqnarray*}
v=0 \spa &\arr& \spa G(u,0) = f(u)-uf'(u)  \spa \arr \spa \max(G(u,0)) = M\\
v=1 \spa &\arr& \spa G(u,1) = (f(u)-uf'(u))(- g'(1)) - 2g'(1) u f'(u) = -g'(1) (f(u)+uf'(u))\\
\spa &\arr& \spa  \max(G(u,1) = -g'(1)
\end{eqnarray*}
Analogously
\begin{eqnarray*}
u=0 \spa &\arr& \spa \max(G(0,v)) = N\\
u=1 \spa &\arr& \spa G(1,v) = -f'(1)
\end{eqnarray*}

Remembering that $-f'(1) = (f-uf')\big|_{\substack{u=1}} \leq M$ and $ - g'(1) = (g-vg')\big|_{\substack{v=1}} \leq N$, we  conclude that
\begin{equation}\label{eq_M1M2}
\underset{\partial S}{ \max } (G) = \max \{ M, N \}\, .
\end{equation}

\end{proof}

Conditions A1,\dots,A3 can be simplified in light of Remark \ref{rem2}:

\begin{itemize}

\item[A1.]  $f(0)=g(0)=1\, , \ f(1) = g(1) = 0$.

\item[A2.] $f$ and $g$ are strictly decreasing functions.

\item[A3.] $fg \leq 1-uv\,.$

\end{itemize}

 So finally the starting point of our investigation is:

\begin{theo}\label{th1}
Let $f$ and $g$ verify A1, A2, A3. 
Then
\begin{center}
$D_\theta$ is a valid copula\spaa $\Longleftarrow  \spaa \theta \in \left[ \dfrac{1}{\al_1}, \dfrac{1}{\al_2} \right]$ 
\end{center}
\end{theo}
\begin{proof}
See Theorem 1 in \cite{KBMS}.
\end{proof}


\section{The restriction on the concavity}	                                    


Assumptions A1, A2, A3  alone cannot guarantee the converse of Theorem \ref{th1} as shown in the following example:

\begin{ex}
Let $f(u) = (1-u)^3$ and $g(v) = (1-v)^3$. The minimum and maximum of $G$ are reached on the diagonal $v=u$ where $G$ has the expression
$$
G(u,u) = (f(u)-uf'(u))^2 - 2u^2 (f'(u))^2 = -\left(-1+u \right)^{4} \left(14 u^{2}-4 u -1\right)
$$
so, according to \eqref{eq_alpha}, we have
\begin{eqnarray*}
\alpha_1 &=& \underset{I}{\min} (G(u,u)) = G \left( \frac47,\frac47 \right) = - \frac{729}{16807}\  \arr\  \frac{1}{\alpha_1} \approx -23.0549\\
\alpha_2 &=& G(0,0) = 1.
\end{eqnarray*}

If Theorem \ref{th1} were reversible, $\theta$ would be constrained to $[-23.0549, 1]$. However, as shown in Figure \ref{fig1}, inequality \eqref{ineq} holds even for $\theta = -30$, indicating that the lower bound is overly restrictive. As a matter of fact, the minimum $\theta$ verifying \eqref{ineq} is $\theta_{\min} \approx -36.1903$, significantly smaller than  $1/\alpha_1$.
\end{ex}

\begin{figure}[htbp]
\begin{center}
\includegraphics[scale=0.32]{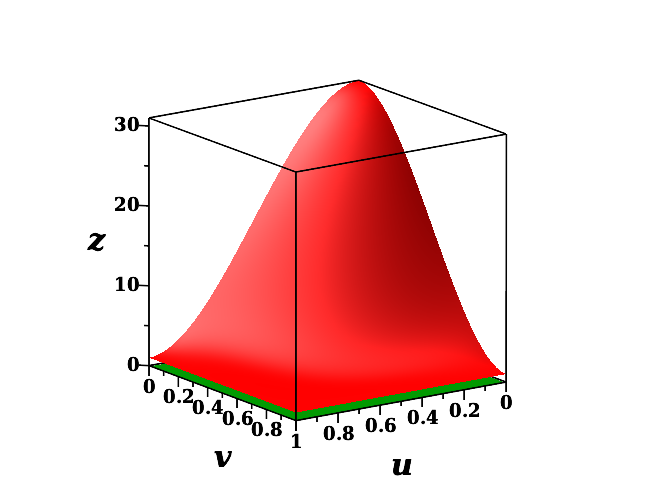}
\includegraphics[scale=0.32]{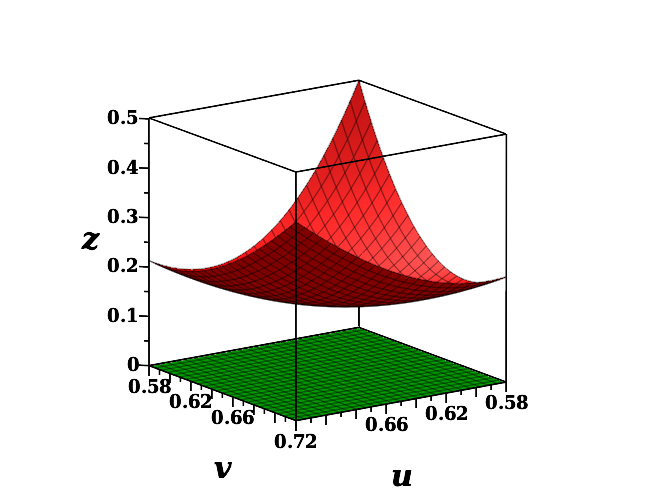}
\caption{Inequality \eqref{ineq} with $\theta=-30$ on $S$ (left) and zoomed around the minimum point (right).}
\label{fig1}
\end{center}
\end{figure}

In order to avoid these situations we restrict our analysis to the case where $f$ and $g$ are concave down. Therefore we will use the following assumptions:

\begin{itemize}

\item[B1.]  $f(0)=g(0)=1\, , \ f(1) = g(1) = 0$.

\item[B2.] $f$ and $g$ are strictly decreasing functions.

\item[B3.]  $f$ and $g$ are concave functions.

\end{itemize}

We introduce now a family of continuous functions $h_M,\ M\geq 1$:
\begin{equation}\label{eq_h}
h_M(u) =
\begin{cases} 
1 & \spa 0 \leq u \leq 1 - \tabfrac{1}{M}, \\
M \cdot (1-u) & \spa  1 - \tabfrac{1}{M} < u \leq 1\, .
\end{cases}
\end{equation}

\begin{rem}\label{rem1}
If $f$ and $g$ verify B1, B2, B3 we have $fg \leq \alpha_2 (1-u v)$.
\end{rem}

\begin{proof}
We assume $M=-f'(1),\ N=-g'(1)$ and, without loss of generality, $\alpha_2=M \geq N$. Because
$$
f \leq h_M \spa \text{and} \spa g \leq h_N
$$
it is enough to prove the result with $f=h_M$ and $g=h_N$.
\begin{itemize}

\item $f=g=1\,.$ 
\begin{eqnarray*}
1 &\leq& -\frac{1}{M} +2\\
1 &\leq& M \left[ 1 - \left( 1-\frac{1}{M} \right)^2 \right]\\
1 &\leq& M \left[ 1 - \left( 1-\frac{1}{M} \right) \left( 1-\frac{1}{N} \right) \right]\\
f g &\leq& M ( 1 - uv )
\end{eqnarray*}

\item $f=1,\ g=N \cdot (1-v)\,.$
\begin{eqnarray*}
N(1-v) &\leq& M(1-v)\\ 
N(1-v) &\leq& M\left[ 1 - \left( 1-\frac{1}{M} \right) v \right]\\ 
N(1-v) &\leq& M ( 1 - u v )\\ 
f g &\leq& M ( 1 - uv )
\end{eqnarray*}

\item $f=M \cdot (1-u),\ g=N \cdot (1-v)\,.$
\begin{eqnarray*}
1 &\leq& \frac{1-uv}{1-u}\\
\frac{1}{1-v} &\leq& \frac{1-uv}{(1-u)(1-v)}\\
N  &\leq& \frac{1-uv}{(1-u)(1-v)}\\
M N(1-u)(1-v) &\leq& M (1-uv)\\
f g &\leq& M ( 1 - uv )
\end{eqnarray*}
where the third inequality follows from $N \leq \dfrac{1}{1-v}$ (see the definition of $h_N$).

\end{itemize}
\end{proof}


\section{A preliminary result}	                                    


The possibility to reverse Theorem \ref{th1} is related to the position of the maximum of $G$ in $S$.

\begin{theo}\label{th2}
Let $D_\theta$ be the copula \eqref{eq_D}. Assume that $f$ and $g$ verify B1, B2, B3 and $
\alpha_2 =  \underset{\partial S}{\max} (G)$. Then
$$
D_\theta\ \text{is a valid copula}\spa  \arr \spa  \theta \in \left[ \frac{1}{\alpha_1}, \frac{1}{\alpha_2}\right]\, .
$$
\end{theo}
\begin{proof}
The maximum of $G$ on $\partial S$, if we  call $a=-f'(1)$ and $b=-g'(1)$, is now 
\begin{equation}\label{eq_al2}
\alpha_2 = \max\{ a,b \}
\end{equation}
because of \eqref{eq_M1M2}.

$D_\theta$ being a valid copula, $\theta$ has to verify \eqref{ineq} for all $(u,v)$ in $S$. In particular:
\begin{eqnarray*}
v=0  &\arr \spa 1 - \theta (f-uf') \geq 0  &\arr \spa \theta \leq \frac1a \\
u=0  &\arr \spa 1 - \theta (g-v g') \geq 0  &\arr  \spa \theta \leq \frac1b
\end{eqnarray*}
so finally $\theta \leq 1 / \max\{ a,b \} =  1 / \alpha_2$ because of \eqref{eq_al2}.
\vspa

The proof for $1/\alpha_1$ is very similar:
$$
G(u,0) = (f(u)-uf'(u))\ \nearrow \spa \arr \spa \underset{I}{\min} (G(u,0)) = 1
$$
and
$$
G(u,1) = b (f(u) + uf'(u)) \ \searrow \spa \arr \spa \underset{I}{\min} (G(u,1)) = -ab
$$
so finally $\min(G)\big|_{\substack{\partial S}} = -ab.$ Now
\begin{equation}\label{eq_al1}
G = fg - vfg' - uf' g -uvf'g' \geq -uvf'g' \geq -ab \spa \arr \spa \alpha_1 = -ab\, .
\end{equation}

Replacing now $u=v=1$ in \eqref{ineq} we get
$$
1 - \theta (-ab) \geq 0 \spa \arr \spa \theta \geq \frac{1}{-ab} = \frac{1}{\alpha_1}
$$
because of \eqref{eq_al1}.
\end{proof}

So the problem is now to identify those functions $f$ and $g$ that generate the function $G$ in \eqref{eq_G} with the property that the maximum is reached on the boundary of $S$. 
%
%
%
%
%
%
%
%
%
%


\section{The restriction on the derivatives}	                                    


In this section, we will impose another constraint on $f'$ and $g'$ to ensure that the maximum of $G$ is on the boundary of $S$ as it is in Theorem \ref{th2}. If $1 \leq \max \{ a, b \} < 2$, then we can construct functions $f$ and $g$ for which the maximum of $G$ is not on the boundary of $S$ and the following is one such example.
 To start, consider the continuous functions \eqref{eq_h} with $M \geq 1$.


For such functions we can observe that we have, in the right neighborhood of $u= 1 - 1/M$:
\begin{equation}\label{eq_M}
u \approx 1-\frac{1}{M}   \spaaa h(u) \approx 1  \spaa -h'(u) = M\, .
\end{equation}

Restricting $G$ to the diagonal of $S$,  and replacing these values of $h$ for both $f$ and $g$, we find:
$$
G(u,u) = (h(u)-uh'(u))^2 - 2u^2 (h'(u))^2 \approx -M^{2}-2 +4M > M =  \underset{\partial S}{\max} (G)
$$
if $1<M<2$ (see Figure \ref{fig2}).

\begin{figure}[htbp]
\begin{center}
\includegraphics[scale=0.4]{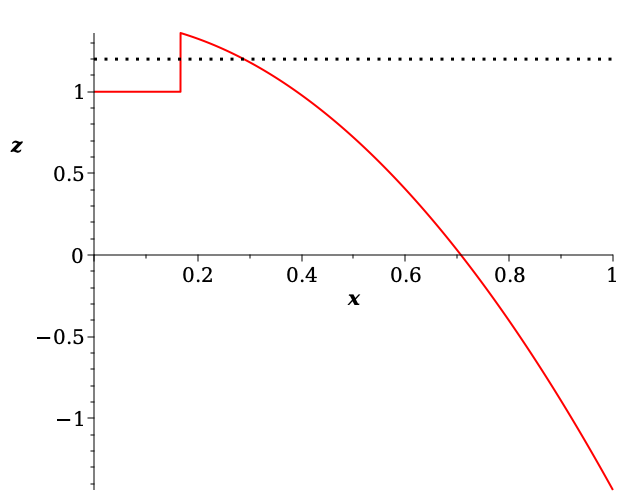}
\caption{The function $z=G(u,u)$ with $f=h_{1.2}$ and z=1.2 (dotted).}
\label{fig2}
\end{center}
\end{figure}

The problem with this (symmetric) example is that $f$ and $-f'$ can simultaneously have values that are very close to their maximum possible values, that are $f \approx 1,\ -f' \approx a$. If we want to rule out such cases we have to introduce an additional bound on the values of the function and its derivative:

%
%

\begin{center}

B4.  $H = f \cdot (1-v g') + g \cdot (1-u f') \leq \max\{ a,b \}+1$.

\end{center}

 Before moving further it could be useful to check if this new condition is reasonable, testing some simple examples where the maximum of $G$ is clearly reached on the boundary of $S$. We use as a benchmark the examples provided in Table 1 in \cite{KBM}.

The results in Table \ref{table} (see Appendix) should convince the reader that B4 is a reasonable choice if we want to get rid of functions like \eqref{eq_h}, while preserving a great variety of well known functions.

\begin{ex}\label{ex2}
 We examine condition B4 on one example:
$$
f(x) = 1-u^n\spaa g(v) = 1-v^m \spaa 1 \leq m \leq n 
$$ 
we have $a = -f'(1) = n,\ b=-g'(1)=m$ and so $\max\{ a,b \}=n$. 
It can be easily checked that there is only one critical point $(u_0,v_0)$ of $H$ in the interior of $S:\ (u_0,v_0) = \left( \left(\dfrac{m-1}{n+m}\right)^{1/n}, \left(\dfrac{n-1}{n+m}\right)^{1/m} \right)$.
A simple calculation gives
$$
H(u_0,v_0) = 2 + \frac{(n-1)(m-1)}{n+m} \leq 2+(n-1) = n+1\, .
$$
On the boundary of $S$ the maximum of $H$ is $H(1,0) = n+1$.
So finally condition B5 reduces to
$$
\underset{S}{\max} (H) = n +1 \leq \max\{ a,b \} +1 = n+1.
$$
\end{ex}

Now, we state the main result of this paper.


\section{The inverse of the theorem}	                                    


We are finally ready to prove the following

\begin{theo}\label{th5}
Let 
$$
D_\theta(u,v) =  \frac{uv}{1 - \theta f(u) g(v)}\,, \spaa 0 \leq u,\, v \leq 1 \, ,\,  \theta \in \R
$$
where $f$ and $g$ verify the following conditions:
\begin{itemize}

\item[B1.]  $f(0)=g(0)=1\, , \ f(1) = g(1) = 0$.

\item[B2.] $f$ and $g$ are strictly decreasing functions.

\item[B3.]  $f$ and $g$ are concave down.

\item[B4.]  $f\cdot (1-v g') + g \cdot (1-u f') \leq \max\{ a,b \}+1$.

\end{itemize}
Then:
\begin{center}
$D_\theta$ is a valid copula\spaa $\Longleftrightarrow  \spaa \theta \in \left[ \dfrac{1}{\al_1}, \dfrac{1}{\al_2} \right]$
\end{center}
\end{theo}

\begin{proof}

\begin{itemize}

\item[$\Longleftarrow)$] We have $\al_2 \geq1$ and so $\theta \leq \dfrac1{\al_2} \leq1$. We start our proof checking that $0 \leq D_\theta \leq1$.
$$
\theta fg \leq fg \spa \Longrightarrow  \spa 1 - \theta fg \geq 1 -  fg \geq 0  \spa \Longrightarrow  \spa D_\theta \geq 0\, .
$$
\begin{eqnarray*}
\text{Remark \ref{rem1}}  \spa &\Longrightarrow&  \spa fg \leq \alpha_2 (1-uv) \spa \Longrightarrow  \spa \frac{fg}{\alpha_2} \leq 1-uv  \spa \Longrightarrow  \spa \theta fg \leq 1-uv\\
&\Longrightarrow&  \spa 1 - \theta fg \geq uv  \spa \Longrightarrow  \spa D_\theta \leq 1\, .
\end{eqnarray*}

The last thing we have to prove is \eqref{ineq}. If $\theta >0$ we have
$$
1- \theta \left[ (f-uf')(g-vg') -2 D_\theta f'g' \right] \geq 1 - \theta G \geq 1 - \theta \alpha_2 \geq 0 
$$
and similarly, if $\theta <0$,
$$
1- \theta \left[ (f-uf')(g-vg') -2 D_\theta f'g' \right] \geq 1 - \theta G \geq 1 - \theta \alpha_1 \geq 0\,.
$$
\vspace{1em}

\item[$\Longrightarrow)$] According to Theorem \ref{th2} we only have to prove that $\alpha_2 =  \underset{\partial S}{\max} (G) = \max\{ a,b \}$. If we remember that $f-uf'$ and $g-vg'$ are increasing functions we immediately realize that their minimum is 1 and so
$$
uf' \leq f-1 \leq 0 \spaaa vg' \leq g-1 \leq 0 \,.
$$

Multiplying the previous  inequalities, we get $uvf'g' \geq (f-1)(g-1)$  that implies
\begin{equation}\label{eq*}
-u v f' g' \leq -fg +f+g-1\, .
\end{equation}

Now 
\begin{eqnarray*}
G &=& (f-uf') (g-vg') -2uv f' g'\\
&\overset{\eqref{eq*}}{\leq}& fg -v f g' -ugf' -fg +f+g-1  \\
&=& \left[ f(1-vg') + g(1-uf') \right] - 1\\
&\overset{(B4)}{\leq}& [\max\{ a,b \} +1] -1 = \max\{ a,b \}\, .
\end{eqnarray*}

\end{itemize}

\end{proof}

\begin{rem}
\begin{itemize}

\item In Example \ref{ex2} we showed that condition B4 reduced to $\max(H) = \max \{a,b\}+1$. This is not a specific feature of this particular example. It is always the case. A simple calculation shows that, if B4 is verified, then:
$$
H(1,0) = a+1\, , \spa H(0,1) = b+1  \spa \arr \spa \max(H) = \max \{ a,b \}+1.
$$

\item Condition A3 is quite restrictive. Note that, at least in the case where $f$ and $g$ are concave functions, it can be omitted (see Theorem \ref{th5}).

\end{itemize}
\end{rem}

\begin{landscape}

\section*{Appendix}

\begin{table}[h!]
\begin{center}
\large 
\begin{tabular}{llcccccc} 
\specialrule{1.5pt}{0pt}{0pt} 
$\boldsymbol{f(u)}$ & $\boldsymbol{g(v)}$ & {\bf Conditions} & {\bf B1} & {\bf B2} & {\bf A3} & {\bf B3} &  {\bf B4}\\
\hline
$1-u^n$ & $1-v^n$ & $1 \leq n \leq 2$ &  T & T& T & T&  T\\
	    &                & $2 < n $          & T &T & {\bf F} &   T & T\\
&&&&&\\
$\log_b(u+b(1-u))$ & $\log_b(v+b(1-v))$ & $1 <b$  & T &T  &T & T& T \\
&&&&&\\
$\cos(\pi u/2)$ & $\cos(\pi v/2)$ &  &T & T  & T & T & T\\
&&&&&\\
$1-u$ & $\log_b(v+b(1-v))$ &  $1 < b \leq 41 $ & T &T &T&T &T \\
 &  &  $41 < b  $ & T&T & {\bf F} &T&T \\
&&&&&\\
$\cos(\pi u/2)$ & $1-v$ &  &T & T & T &  T & T\\
&&&&&\\
$\log_b(v+b(1-v))$ & $\cos(\pi v/2)$ & $1 < b$ & T & T& T& T &T\\
&&&&&\\
$(1-u) e^{cu}$ & $(1-v) e^{cv}$ & $0 \leq c \leq 1$ & T & T & T & T & T\\
&&&&&\\
$\dfrac{{\mathrm e}^{a u}-{\mathrm e}^{a}}{1-{\mathrm e}^{a}}$ & $\dfrac{{\mathrm e}^{a v}-{\mathrm e}^{a}}{1-{\mathrm e}^{a}}$  & $0 < a \leq 3.7$ & T  & T & T & T & T\\
 &  & $3.7 < a $ & T & T & {\bf F} & T  & T\\
&&&&&\\
$h_M(u)$ & $h_M(v)$ & 1<M<2 &T & T & T&T &{\bf F}\\
\specialrule{1.5pt}{0pt}{0pt} 
\end{tabular}
\end{center}
\caption{Testing conditions B1,\dots,B4 (T=True, F=False). We added A3 to show how restrictive this condition is. }
\label{table1}
\label{table}
\end{table}

\end{landscape}

\newpage


\end{document}